\documentclass[12pt]{amsart}
\usepackage{amsmath}
\usepackage{amsfonts}
\usepackage{hyperref}
\usepackage{amssymb}
\usepackage{latexsym}
\usepackage{amsthm}
\usepackage{a4wide}

\usepackage{graphicx}

\usepackage[ansinew]{inputenc}
\usepackage{amsfonts,epsfig}
\usepackage{mathabx}

\newtheorem{theorem}{Theorem}
\newtheorem{lemma}[theorem]{Lemma}

\newtheorem{coro}[theorem]{Corollary}

\newtheorem{proposition}[theorem]{Proposition}

\newcounter{other}            

\theoremstyle{definition}

\DeclareMathAccent{\widecheck}{0}{mathx}{"71}

\makeatletter
\newcommand*{\rom}[1]{\expandafter\@slowromancap\romannumeral #1@}
\makeatother

\numberwithin{equation}{section}
\begin{document}

\title[Hankel forms and measures]{Hankel forms and measures on weighted Bergman spaces}

\author[Setareh Eskandari]{Setareh Eskandari}
\address{Setareh Eskandari  \\Department of Mathematics and Mathematical Statistics \\
	Ume{\aa} University\\
	90187 Ume{\aa}\\
	SWEDEN.}\email{setareh.eskandari@umu.se}

\author[Antti Per\"{a}l\"{a}]{Antti Per\"{a}l\"{a}}
\address{Antti Per\"{a}l\"{a} \\Department of Mathematics and Mathematical Statistics \\
	Ume{\aa} University\\
	90187 Ume{\aa}\\
	SWEDEN.} \email{antti.perala@umu.se}



%
\subjclass[2020]{Primary 47B35; Secondary 30H05, 46E20.}

\keywords{Hankel form, Hankel operator, Hankel measure, Bergman space, Doubling weight}

\thanks{S. Eskandari was funded by the postdoctoral scholarship JCK22-0052 from the Kempe Foundations.}


\begin{abstract}
We characterize the boundedness of Hankel forms and Hankel operators induced by measures on weighted Bergman spaces, where the weights satisfy an upper-doubling condition. We also characterize $A^p_\omega$ Hankel measures for $p\leq 2$. The proofs leverage the existing theory of weighted Bergman spaces and the recent results on two-weight fractional derivatives, also simplifying the recent $A^1$ duality for small Bergman spaces obtained by Pel\'aez and R\"atty\"a.
\end{abstract}

\maketitle



\section{Introduction}

\noindent Let $\mathbb{D}=\{|z|<1\}$ denote the open unit disk of the complex plane and $dA(z)=\pi^{-1}dxdy$, where $z=x+iy$, denote the normalized Lebesgue area measure on $\mathbb{D}$. We denote by $\mathcal{H}(\mathbb{D})$ the space of analytic functions on $\mathbb{D}$ and by $\mathcal{H}(\overline{\mathbb{D}})$ the space of functions that are analytic on an open set containing the closure $\overline{\mathbb{D}}$.

Given a complex Borel measure $\mu$ on the disk, we define the bilinear Hankel form $\textbf{H}_{\mu}$ by the formula
$$\textbf{H}_{\mu}(f,g)=\int_{\mathbb{D}}fgd\mu,$$
and call $\mu$ the symbol of $\textbf{H}_{\mu}$. Observe that this is well defined if $f,g \in \mathcal{H}(\overline{\mathbb{D}})$.

Let $\omega$ be a radial weight, which means in this paper that it is a non-negative, integrable function with $\omega(z)=\omega(|z|)$ for $z \in \mathbb{D}$. While $\omega$ is allowed to take the value $0$, we assume that it carries some mass near the boundary of the disk in the following sense: $\widehat{\omega}(\rho)>0$ for every $\rho \in [0,1)$, where $\widehat{\omega}(\rho)=\int_\rho^1 \omega(t)dt$. We denote $d\omega=\omega dA$, and for $0<p<\infty$ define the weighted Bergman spaces $A^p_\omega$ to consist of functions $f\in\mathcal{H}(\mathbb{D})$ with
$$\|f\|_{p,\omega}=\left(\int_{\mathbb{D}}|f|^p d\omega\right)^{1/p}<\infty.$$
In other words $A^p_\omega=\mathcal{H}(\mathbb{D})\cap L^p_\omega$. The assumptions on $\omega$ guarantee that the spaces $A^p_\omega$ are complete and that for every $z\in \mathbb{D}$ there exists a reproducing kernel $B_z^\omega \in A^2_\omega$ with
$$\langle f,B_z^\omega\rangle_\omega=\int_{\mathbb{D}}f\overline{B_z^\omega}d\omega=f(z),\quad f \in A^2_\omega.$$ The functions $B_z^\omega$ also give rise to the Bergman projection $P_\omega$, which is defined by
$$P_\omega[f](z)=\int_{\mathbb{D}}f\overline{B_z^\omega}d\omega \quad \text{and}\quad  P_\omega[\mu](z)=\int_{\mathbb{D}}\overline{B_z^\omega}d\mu$$
for $f \in L^1_\omega$ and $\mu$ a complex Borel measure, respectively. Observe that by our assumptions on $\omega$, the functions $B_z^\omega$ actually belong to $\mathcal{H}(\overline{\mathbb{D}})$, so the integrals are well-defined.

We can now define the (small) Hankel operator $H_\mu^\omega$ by
$$H_\mu^\omega f(z)=\int_{\mathbb{D}}fB_z^\omega d\mu,$$
which is well-defined for every $f \in \mathcal{H}(\overline{\mathbb{D}})$.

The Hankel forms and Hankel operators are closely connected to many questions in analysis. We mention Toeplitz operators, factorization, duality and the polynomially bounded operator problem \cite{AC, DRWW, JPR, PZ, pisier, tolo}. It is often natural to restrict the study to the case of anti-analytic symbols. Some works without this restriction do exist; we mention \cite{PRpreprint, TasVir, zhzh}.

In this paper we focus on the problem of characterizing the boundedness $\textbf{H}_{\mu}$ and $H_\mu^\omega$ in the context of the weighted Bergman spaces $A^p_\omega$, where the weight $\omega$ satisfies the conditions given above, and
$$\widehat{\omega}(\rho)\lesssim \widehat{\omega}\left(\frac{1+\rho}{2}\right),\quad \rho \in [0,1).$$
We denote by $\widehat{\mathcal{D}}$ the collection of all such weight and call them upper-doubling. If the weight $\omega$ satisfies
$$\widehat{\omega}(\rho)\geq C\widehat{\omega}\left(1-\frac{1-\rho}{K}\right)$$
for some $C,K>1$, we write $\omega \in \widecheck{\mathcal{D}}$ and call $\omega$ lower-doubling. Finally, we denote $\mathcal{D}=\widehat{\mathcal{D}}\cap \widecheck{\mathcal{D}}$.

These weights have been studied extensively in pioneering works of Pel\'aez and R\"atty\"a, and by several other authors \cite{DRWW, tan, 3, PPRhankel, PdlR, PdlR2, 2, PRkernel, 1, PRpreprint, antti2, antti1}. Note that $\mathcal{H}(\overline{\mathbb{D}})$ is dense in $A^p_\omega$, and we say that $\textbf{H}_{\mu}:A^p_\omega\times A^q_\omega\to \mathbb{C}$ is bounded if
$$|\textbf{H}_{\mu}(f,g)|\lesssim \|f\|_{p,\omega}\cdot \|g\|_{q,\omega},\quad f,g \in \mathcal{H}(\overline{\mathbb{D}}).$$
Likewise, we say that $H_\mu^\omega:A^p_\omega \to \overline{A^q_\omega}$ is bounded, if
$$\|H_\mu^\omega f\|_{q,\omega}\lesssim \|f\|_{p,\omega}, \quad f \in \mathcal{H}(\overline{\mathbb{D}}).$$
We also define $\|\textbf{H}_\mu\|_{A^p_\omega \times A^q_\omega}$ as the supremum of $|\textbf{H}_{\mu}(f,g)|$ where $f,g \in \mathcal{H}(\overline{\mathbb{D}})$ are taken from the unit balls of $A^p_\omega$ and $A^q_\omega$, respectively. The norm $\|H_\mu^\omega\|_{A^p_\omega \to A^q_\omega}$ is the supremum of $\|H_\mu^\omega f\|_{q,\omega}$, where $f \in \mathcal{H}(\overline{\mathbb{D}})$ is taken from the unit ball of $A^p_\omega$.

For a number $s\in (1,\infty)$ we let $s'$ denote its H\"older conjugate $s'=s/(s-1)$. Given a radial weight $\omega$ and $x\geq 1$, we denote by $W$ the following weight
\begin{equation}\label{W}
W(z)=W_{x,\omega}(z)=(x-1)\widehat{\omega}(z)^x(1-|z|)^{x-2}+x\omega(z)\widehat{\omega}(z)^{x-1}(1-|z|)^{x-1}.
\end{equation}
We are now ready to state the first two main results of this paper.

\begin{theorem}\label{TH1}
Let $\mu$ be a complex Borel measure on $\mathbb{D}$, $ \omega \in \widehat{\mathcal{D}}$ and $0<p,q<\infty$, and let $r$ satisfy $1/p+1/q=1/r$. Then $\textbf{H}_{\mu}:A^p_\omega \times A^{q}_\omega \to \mathbb{C}$ is bounded bilinear form if and only if
	
	\begin{itemize}
	
	\item[\rom{1}:]  $ P_\omega(\bar{\mu}) $ belongs to $A^{r'}_\omega $, if $r>1$.
	
	\item[\rom{2}:]  $ P_\omega(\bar{\mu})  $  belongs to ${\mathsf{D}}_{\omega} \, BMOA(\infty,\omega)$, if $r=1$.
	
	\item[\rom{3}:]  $ P_W(\bar{\mu})  $ belongs to $\mathcal{B}$, if $r<1$.
	
	\end{itemize}
	
In \rom{3} the weight $W=W_{1/r,\omega}$ is given in \eqref{W}. We also have that $\|\textbf{H}_\mu\|_{A^p_\omega \times A^q_\omega}\asymp \|P_\omega(\bar{\mu})\|_X$, where $X$ is the space in question depending on the case.
	
\end{theorem}

In the above theorem, $\mathcal{B}$ denotes the Bloch space, consisting of analytic functions $f$ on $\mathbb{D}$ with
$$\|f\|_{\mathcal{B}}=\sup_{z \in \mathbb{D}}(1-|z|^2)|f'(z)| + |f(0)|<\infty.$$
The space ${\mathsf{D}}_{\omega} \, BMOA(\infty,\omega)$ is a fractional integral of a certain $BMOA$ type spaces -- its definition is more complicated and postponed until Section 3.

The second main theorem could be equally well taken as a corollary to Theorem \ref{TH1}. The three cases can be read directly from it. However, note that here we require $1<q<\infty$, whereas in Theorem \ref{TH1} any positive $q$ is allowed.

\begin{theorem}\label{TH1b}
Let $\mu$ be a complex Borel measure on $\mathbb{D}$, $ \omega \in \widehat{\mathcal{D}}$ and $0<p<\infty$, $1<q<\infty$. Then $H_\mu^\omega : A^p_\omega \to \overline{A^q_\omega}$ is bounded if and only if $\textbf{H}_{\mu}:A^p_\omega \times A^{q'}_\omega \to \mathbb{C}$ is a bounded bilinear form. We also have that $\|H_\mu^\omega\|_{A^p_\omega \to A^q_\omega}\asymp \|P_\omega(\bar{\mu})\|_X$, where $X$ is the space in question depending on which case from Theorem \ref{TH1}.
\end{theorem}

The proofs of the above theorems are based on duality and factorization of these weighted Bergman spaces, and are in the spirit of a proof that was already considered well-known in the paper of Tolokonnikov \cite{tolo}. 

Since there have been a number of papers with results close to ours, a comparison is in order. In \cite{zhzh} Zhou studied essentially the same Hankel operator and form and obtained cases \rom{2} and \rom{3} when $\omega(z)=(\alpha+1)(1-|z|^2)^\alpha$ with $\alpha>-1$ is a standard weight. In \cite{tan} Korhonen and R\"atty\"a obtain our case \rom{1} with the same weights, but for $d\mu=\overline{f}d\omega$ with $f$ analytic. In \cite{DRWW} and with the same restriction on the symbols, all the cases presented here are considered under the condition $\omega \in \mathcal{D}$, but with possibly different weights for the domain and the target space. In the very recent paper \cite{PRpreprint}, the authors consider the case \rom{2} for $d\mu=fd\omega$ with $f \in L^1_{\omega log}\subset L^1_\omega$ and $\omega \in \mathcal{D}$. In that paper the result is given in terms of an operator $V_{\omega,\nu}$, but this operator is in fact a fractional derivative operator.

The advantage of the present work is that we allow $\omega \in \widehat{\mathcal{D}}$ and the symbol class of complex Borel measures is the most general among the papers discussed. By using some recent results on two-weight fractional derivatives, we also think that our proofs are considerably simpler than many of the proofs from the mentioned papers. The case \rom{2} for $\omega\in \widehat{\mathcal{D}}$ is probably the most novel, and would probably be out of reach without the use of fractional derivatives. 

If $\omega \in \mathcal{D}$, then case \rom{2} becomes more clear, as ${\mathsf{D}}_{\omega} \, BMOA(\infty,\omega)$ becomes the Bloch space due to the duality result in \cite{1}. In fact, the cases \rom{2} and \rom{3} can be written in an simpler way if the weight in question is more regular. We also point out the recent paper \cite{zhzh}, where slightly different Hankel operators and forms were studied for Hardy spaces and standard weighted Bergman spaces. In the context of the present paper, these would be
$$\widetilde{H}_\mu^\omega f=\int_{\mathbb{D}}fB^\omega_{\overline{z}}(\xi)d\omega(\xi)\quad \text{ and } \quad \widetilde{\textbf{H}}_\mu(f,g)=\int_{\mathbb{D}}f\widetilde{g}d\mu,$$
where $\widetilde{g}(z)=\overline{g(\overline{z})}$. Since $g\mapsto \widetilde{g}$ is an isometric isomorphism of $A^p_\omega$ for any radial $\omega$, Theorem \ref{TH1} holds for these objects with little or no changes. 

The Hankel form and Hankel operator defined above are also related to the problem of Hankel measures, also called pseudo-Carleson measures. A complex Borel measure $\mu$ is called a Hankel measure for $A^p_\omega$ (which is not a standard terminology) if
$$\left|\int_{\mathbb{D}}f^2 d\mu\right|\lesssim \|f\|_{p,\omega}^2$$
for every $f \in \mathcal{H}(\overline{\mathbb{D}})$. This kind of measures were first studied by Xiao \cite{xi1, xi2}, who mentions that he learned about the problem from Peetre. Recently, there have been a number of papers on the topic, we mention \cite{ARSW, bao, LH, SY, wang, zorboska}.

Observe that for the integral above to exist unconditionally for every $f \in A^2_\omega$, we would need to have that the total variation $|\mu|$ is Carleson measure for $A^p_\omega$, which would trivialize the question. Our third main theorem characterizes $p$-Hankel measures. When $p=2$, we are admittedly restricted to the case $\omega \in \mathcal{D}$ due to our poor understanding of the space ${\mathsf{D}}_{\omega} \, BMOA(\infty,\omega)=\mathcal{B}$.

\begin{theorem}\label{TH2}
Let $\mu$ be a complex Borel measure and $\omega \in \widehat{\mathcal{D}}$. Then $\mu$ is a Hankel measure for $A^p_\omega$ if and only if
\begin{itemize}
\item[\rom{1}:] $P_W(\overline{\mu}) \in \mathcal{B}$, if $p<2$;
\item[\rom{2}:] $P_\omega(\overline{\mu})\in \mathcal{B}$, if $p=2$ and $\omega \in \mathcal{D}$.
\end{itemize}
In \rom{1} the weight $W=W_{2/p,\omega}$ is given in \eqref{W}.
\end{theorem}

We note that the sufficiency in Theorem \ref{TH2} clearly follows from duality for the space $A^{p/2}_\omega$, as for $\omega \in \mathcal{D}$ one has ${\mathsf{D}}_{\omega} \, BMOA(\infty,\omega)=\mathcal{B}$. The necessity follows from taking a suitable test function and using the recent fractional derivative characterizations of the Bloch space \cite{PRW}. Also in the case $p>2$, duality easily gives a sufficient condition. As for the necessity, one usually proves such estimates with atomic decomposition and Kahane-Khinchine inequalities. We expect this to be the case here as well, but regret to admit to not having succeeded in this.

The paper is structured as follows. In Section 2, we discuss the basic properties of fractional derivatives. These can be used to simplify our main theorem. In Section 3, we study the space ${\mathsf{D}}_{\omega} \, BMOA(\infty,\omega)$, which is involved in the case \rom{2} of the main theorems obtaining a dual pairing for $A^1_\omega$, which is easier to understand than that of \cite{1}. Section 4 deals with Hankel forms and Hankel operators and contains the proof of Theorem \ref{TH1}. Section 5 is dedicated to Hankel measures and the proof of Theorem \ref{TH2}. In the fifth and last section of the paper, we discuss how Theorem \ref{TH1} can be simplified when $\omega \in \mathcal{D}$, and further when $\omega(z)=(\alpha+1)(1-|z|^2)^\alpha$ ($\alpha>-1$) is a standard weight, allowing us to compare with results in \cite{zhzh}. In that section, we also consider some possible generalizations.

Finally, a word about notation. By $A(x)\lesssim B(x)$ we mean that there exists a constant $C>0$, not depending on $A$ and $B$ so that $A(x)\leq B(x)$ for every $x$. If both $A(x)\lesssim B(x)$ and $B(x)\lesssim A(x)$ hold, we write $A(x)\asymp B(x)$.

\bigskip

\section{Fractional derivatives}

\noindent Kernel functions induced by radial weights give rise to fractional derivatives. In \cite{zhu1}, Zhu studied such operators, and their theory was further streamlined in the book \cite{zhuholom}. In fact, this concept has a rather long history that can be traced back to the 1932 work of Hardy and Littlewood \cite{HL}, and further to 1769 work of Euler, see \cite{PRW}. By denoting by
$$\omega_x = \int_0^1 \omega(s)s^xds$$
the $x$ moment of the radial weight $\omega$, we define for $f(z)=\sum_{n=0}^\infty f_nz^n$:
$$\mathsf{R}^{\omega,\nu}f(z)=\sum_{n=0}^\infty \frac{\omega_{2n+1}}{\nu_{2n+1}}f_n z^n$$
and
$$\mathsf{D}^\omega f(z)=\sum_{n=0}^\infty \frac{1}{\omega_{2n+1}}f_n z^n,\quad \mathsf{D}_\omega=\sum_{n=0}^\infty \omega_{2n+1}f_nz^n.$$
The operator $\mathsf{R}^{\omega,\nu}$ is the two-weight fractional derivative, which has been studied in \cite{antti2, PRW}. The operators $\mathsf{D}^\omega$ and $\mathsf{D}_\omega$ were studied in \cite{PdlR, PdlR2} by Pel\'aez and de la Rosa.

Since $\lim_{n\to \infty}(\omega_{2n+1})^{1/n}=1$, all operators above are bijections on $\mathcal{H}(\mathbb{D})$ and $\mathcal{H}(\overline{\mathbb{D}})$, and satisfy $(\mathsf{R}^{\omega,\nu})^{-1}=\mathsf{R}^{\nu,\omega}$, $(\mathsf{D}^\omega)^{-1}=\mathsf{D}_\omega$, $\mathsf{R}^{\omega,\nu}=\mathsf{D}^\nu \mathsf{D}_\omega= \mathsf{D}^\omega \mathsf{D}_\nu$. By denoting 
$$\omega_{+}(\rho)=\int_{\rho}^1 \omega(s)\frac{ds}{s},$$
one can also show that $\mathsf{D}^\omega=\mathsf{R}^{1,\omega_+}$, where $1$ in this formula is simply the constant weight equal to one.

The key property of the operator $\mathsf{R}^{\omega,\nu}$ is that
$$\mathsf{R}^{\omega,\nu} B_z^\omega=B_z^\nu,$$
and that
$$\mathsf{R}^{\omega,\nu}P_\omega(\mu)=P_\nu(\mu),$$
for any complex Borel measure $\mu$. These will be useful for our main theorem.

Recently there have been several characterizations of the Bloch space in terms of fractional derivatives, see \cite{PdlR, PRW}. The following result, which is Theorem 1 in \cite{PRW} will be especially useful for us.

\begin{theorem}[Per\"al\"a, R\"atty\"a and Wang \cite{PRW}]\label{bfrac}
Let $\omega \in \mathcal{D}$ and $\nu \in \widehat{\mathcal{D}}$. Then for $f \in \mathcal{H}(\mathbb{D})$ is holds that
$$\|f\|_{\mathcal{B}}\asymp \sup_{z \in \mathbb{D}}\frac{\widehat{\nu}(z)}{\widehat{\omega}(z)}|\mathsf{R}^{\omega,\nu}f(z)|$$
if and only if
$$\int_0^{\rho}\frac{\widehat{\omega}(t)}{\widehat{\nu}(t)(1-t)}dt\lesssim \frac{\widehat{\omega}(\rho)}{\widehat{\nu}(\rho)}.$$
\end{theorem}

\bigskip

\section{The space $BMOA(\infty,\omega)$}

\noindent If $\omega \in \mathcal{D}$, it is well-known that the dual of $A^1_\omega$ equals the Bloch space under the standard dual pairing. If $\omega \in \widehat{\mathcal{D}}$, the problem is much more subtle. An analytic $f$ belongs to the space $BMOA(\infty,\omega)$ if
$$\|f\|_{BMOA(\infty,\omega)}=\sup_{0<\rho<1}\|f_\rho\|_{BMO}\cdot\widehat{\omega}(\rho)<\infty.$$
Here $\|\cdot\|_{BMO}$ is the classical $BMO$ norm appearing in Fefferman's $H^1$ duality, but we will not need its exact formula in this paper. In \cite{1} the dual of $A^1_\omega$ is shown to be isomorphic to $BMOA(\infty,\omega)$ under the rather curious pairing
$$\langle f,g\rangle_{\omega \circ \omega}=\lim_{\rho\to 1^-}\sum_{n=0}^\infty f_n \overline{g_n}\rho^{2n+1}(\omega_{2n+1})^2,$$
whenever the limit exists. Note that 
$$\int_{\mathbb{D}}f_\rho\overline{g_\rho}d\omega=\sum_{n=0}^\infty f_n\overline{g_\rho}\rho^{2n}\omega_{2n+1}.$$
Using $\mathsf{D}_\omega$ on $g$, this leads to
$$\langle f,g\rangle_{\omega \circ \omega}=\lim_{\rho\to 1^-}\rho \int_{\mathbb{D}}f_\rho\overline{(\mathsf{D}_\omega g)_\rho}d\omega.$$
Note that if $g \in BMOA(\infty,\omega)$, then $g$ is contained in $P_\omega(L^\infty)\subset \mathcal{B}$. So, if also $f \in \mathcal{H}(\overline{\mathbb{D}})$, then certainly
$$\langle f,g\rangle_{\omega \circ \omega}=\int_{\mathbb{D}}f\overline{(\mathsf{D}_\omega g)}d\omega.$$
Let $\mathsf{D}_\omega BMOA(\infty,\omega)$ denote the image of $BMOA(\infty,\omega)$ under $\mathsf{D}_\omega$. It can be equipped with the norm
$$\|g\|_{\mathsf{D}_\omega BMOA(\infty,\omega)}=\|\mathsf{D}^\omega g\|_{BMOA(\infty,\omega)}+|g(0)|.$$
This leads to the following alternative characterization of $(A^1_\omega)^*$ with $\omega \in \widehat{\mathcal{D}}$ under the standard pairing.

\begin{proposition}\label{A1}
Let $\omega \in \widehat{\mathcal{D}}$. Then the dual of $A^1_\omega$ equals $\mathsf{D}_\omega BMOA(\infty,\omega)$ under the pairing
$$\lim_{\rho\to 1^-}\int_{\mathbb{D}}f_\rho\overline{g}d\omega=\lim_{\rho\to 1^-}\int_{\mathbb{D}}f\overline{g_\rho}d\omega,\quad f\in A^1_\omega, g \in \mathsf{D}_\omega BMOA(\infty,\omega).$$
\end{proposition}

It is worth remarking that if $\omega \in \mathcal{D}$, then the dual of $A^1_\omega$ is $\mathcal{B}$ under the pairing of Proposition \ref{A1} (see Theorem 3 of \cite{1}). This gives us the interesting identity $\mathsf{D}_\omega BMOA(\infty,\omega)=\mathcal{B}$, or equivalently $BMOA(\infty,\omega)=\mathsf{D}^\omega \mathcal{B}$, for $\omega \in \mathcal{D}$. Further specializing to standard weights, this leads to characterizations for the Bloch spaces in terms of dilated BMO norms.

\bigskip

\section{Hankel forms and Hankel operators}

\noindent We begin with the following lemma.

\begin{lemma}
Let $\mu$ be a complex Borel measure, $\nu$ be a radial weight and $F \in \mathcal{H}(\overline{\mathbb{D}})$. Then,
$$\int_{\mathbb{D}}Fd\mu=\lim_{\rho\to 1^-}\int_{\mathbb{D}}F\overline{(P_\nu(\overline{\mu}))_\rho}d\nu.$$
\end{lemma}

\begin{proof}
We observe that $P_\nu(\overline{\mu})$ is a well-defined analytic function, and
\begin{align*}
	&\int_{\mathbb{D}}F\overline{(P_\nu(\overline{\mu}))_\rho}d\nu\\
	=&\int_{\mathbb{D}}F(z)\left(\int_{\mathbb{D}}B_{\rho z}^{\nu}(\xi)d\mu(\xi)\right)d\nu(z)\\
	=&\int_{\mathbb{D}}\left(\int_{\mathbb{D}}F(z)\overline{B_{\rho\xi}^\nu(z)}d\nu(z)\right)d\mu(\xi)\\
	=&\int_{\mathbb{D}}F_\rho d\mu.
\end{align*}
Here we used that for a radial $\nu$ one has $B^\nu_{\rho z}(\xi)=\overline{B^\nu_{\rho \xi}(z)}$. The change in the order of the integration is legal, since the kernel functions $B_{\rho z}^\nu$ are uniformly bounded on the disk; $|B_{z\rho}^\nu(\xi)|\leq C$ for every $z,\xi \in \mathbb{D}$. Since $F_\rho\to F$ uniformly when $\rho\to 1^-$ and $F\in \mathcal{H}(\overline{\mathbb{D}})$, we obtain the desired identity.

\end{proof}

Note that since $P_\nu(\overline{\mu})$ might be outside $A^1_\nu$, we cannot remove the dilatation from it in general, even if $F \in \mathcal{H}(\overline{\mathbb{D}})$. However, in many cases $P_\nu(\overline{\mu}) \in A^1_\omega$, and the dilatation can be removed. The left-hand side of the identity is always well-defined for $F$ as above.

From this lemma it easily follows that if $f,g \in \mathcal{H}(\overline{\mathbb{D}})$, then
\begin{equation}\label{hform}
\textbf{H}_{\mu}(f,g)=\lim_{\rho\to 1^-}\int_{\mathbb{D}}fg\,\overline{(P_\omega(\overline{\mu}))_\rho}d\omega.
\end{equation}

Similarly, one obtains that
\begin{equation}\label{hoperator}
\lim_{\rho\to 1^-}\langle (H_\mu^\omega f)_{\rho},\overline{g}\rangle_\omega=\textbf{H}_{\mu}(f,g),
\end{equation}
which is a well-known characteristic of a Hankel operator.

In order to smoothly prove Theorem \ref{TH1}, we would need a factorization theorem for Bergman spaces. Fortunately, it was recently proven by Korhonen and R\"atty\"a \cite{tan}.

\begin{theorem}[Korhonen and R\"atty\"a \cite{tan}]\label{KR}
Let $\omega \in \widehat{\mathcal{D}}$ and $0<p,q,r<\infty$ satisfying $1/r=1/p+1/q$. Then for every $F \in A^r_\omega$, there exist $f \in A^p_\omega$ and $g \in A^q_\omega$ so that $F=fg$, and
$$\|F\|_{r,\omega}\asymp \|f\|_{p,\omega}\cdot \|g\|_{q,\omega}.$$
\end{theorem}

\noindent \emph{Proof of Theorems \ref{TH1} and \ref{TH1b}:} 
Observe that the boundedness of the Hankel form (as given in this paper) is equivalent with
$$|\textbf{H}_\mu (f,g)|\lesssim \|f\|_{p,\omega}\cdot\|g\|_{q,\omega}$$
for $f \in A^p_\omega$ and $g \in A^q_\omega$ for which $fg \in \mathcal{H}(\overline{\mathbb{D}})$. Indeed, one direction is obvious and the other follows by approximation.

Let now $F \in \mathcal{H}(\overline{\mathbb{D}})\subset A^r_\omega$. By Theorem \ref{KR}, $F$ can be written as $fg$ with $f \in A^p_\omega$ and $g \in A^{q}_\omega$ and
$$\|F\|_{r,\omega}\asymp \|f\|_{p,\omega}\cdot \|g\|_{q,\omega}.$$

Observe that the theorem does not seem to give that also $f$ and $g$ can be chosen to be in $\mathcal{H}(\overline{\mathbb{D}})$, but by the remark above, it is not a problem.

Now the boundedness of both the Hankel form is reduced to the boundedness of the functional
$$F\mapsto \lim_{\rho\to 1^-}\int_{\mathbb{D}}F\,\overline{(P_\omega(\overline{\mu}))_\rho}d\omega$$
on $A^r_\omega$ that is, to duality of these spaces.

The case \rom{1} is standard application of duality $(A^r_\omega)^*=A^{r'}_\omega$, see \cite{1}. For the case \rom{3}, note that \eqref{hform} is true for any radial weight. In \cite{antti1}, it was proven that if $0<r<1$, then the dual of $A^r_\omega$ is equal to the Bloch space under the pairing
$$\int_{\mathbb{D}}f\overline{g}dW,$$
where $W$ is as in the statement of the theorem. 

For the case \rom{2} that is when $r=1$, we use the duality result for $A^1_\omega$ obtained in Proposition \ref{A1}. The result follows immediately.

Regarding Theorem \ref{TH1b}. Assume now $1<q<\infty$ and use duality $(A^q_\omega)^*=A^{q'}_\omega$, which was proven for upper doubling weights in \cite{1}. Equations \eqref{hform} and \eqref{hoperator} give that boundedness of $H_\mu^\omega : A^p_\omega \to \overline{A^q_\omega}$ is equivalent with the boundedness of the Hankel form $\textbf{H}_\mu:A^p_\omega \times A^{q'}_\omega \to \mathbb{C}$. Theorem \ref{TH1b} now is a particular case of Theorem \ref{TH1}. The norm estimates follow from the same estimates for the dual spaces. The proof of both theorems is complete.

\bigskip

\section{Hankel measures}

\noindent We now turn to our result on Hankel measures. Observe that if $f \in A^p_\omega$, then $f^2 \in A^{p/2}_\omega$ and
$$\int_{\mathbb{D}}f^2 d\mu=\int_{\mathbb{D}} f^2 \overline{P_\nu(\overline{\mu})}d\nu$$ for any radial weight $\nu$ with $\widehat{\nu}>0$. The sufficiency in Theorem \ref{TH2} comes therefore immediately from the corresponding dualities with $\nu=\omega$ when $p=2$ and $\nu=W$ with $W=W_{p/2}$ for $p<2$. Similarly, the sufficiency for $p>2$ would also follow.

To prove the necessity, we only need the estimates for the norms of the kernel function from \cite{PRkernel}. The following theorem is a simplified variant of Theorem 1 and Corollary 2 of that paper.

\begin{theorem}[Pel\'aez and R\"atty\"a \cite{PRkernel}]\label{kernel}
Let $\omega,\nu \in \widehat{\mathcal{D}}$. Then
$$\|B_z^\nu\|_{p,\omega}^p\asymp \int_0^{|z|}\frac{\widehat{\omega}(t)}{\widehat{\nu}(t)^p (1-t)^p}dt,\quad |z|\to 1^-.$$
Moreover, if 
$$\int_0^{\rho}\frac{\widehat{\omega}(t)}{\widehat{\nu}(t)^p(1-t)^p}dt\lesssim \frac{\widehat{\omega}(\rho)}{\widehat{\nu}(\rho)^p(1-\rho)^{p-1}},\quad \rho\to 1^-,$$
then
$$\|B_z^\nu\|_{p,\omega}^p\asymp \frac{\widehat{\omega}(|z|)}{\widehat{\nu}(|z|)^p(1-|z|)^{p-1}},\quad |z|\to 1^-.$$
\end{theorem}

\noindent \emph{Proof of necessity in Theorem \ref{TH2}:} Let $F_z(\xi)=(1-\overline{z}\xi)^{-2-\beta}$, with $\beta>-1$. This is the reproducing kernel induced by the weight $\nu(z)=(\beta+1)(1-|z|^2)^\beta$. Observe that $\widehat{\nu}(z)\asymp (1-|z|)^{1+\beta}$. Since $\omega \in \widehat{\mathcal{D}}$ it there exists (see \cite{PRkernel}) $\gamma>0$ such that
$$\frac{\widehat{\omega}(\rho_1)}{(1-\rho_1)^\gamma}\lesssim \frac{\widehat{\omega}(\rho_2)}{(1-\rho_2)^\gamma},\quad 0\leq \rho_1<\rho_2<1.$$

Let now $\beta$ be chosen so that $p(\beta+1)+p>\gamma+1$ (that is, $\beta$ is big enough). Then for $t<\rho$ one has
$$\frac{\widehat{\omega}(t)}{\widehat{\nu}(t)^p(1-t)^p}\asymp\frac{\widehat{\omega}(t)}{(1-t)^{(p(\beta+1)+p-\gamma-1)+\gamma+1}}\lesssim \frac{\widehat{\omega}(\rho)}{(1-\rho)^{\gamma}}\frac{\widehat{\omega}(t)}{(1-t)^{(p(\beta+1)+p-\gamma-1)+1}},$$
so the latter part of Theorem \ref{kernel} gives
$$\|F_z\|_{A^p_\omega}^2\asymp \frac{\widehat{\omega}(|z|)^{2/p}}{(1-|z|)^{2\beta+4-2/p}},$$
when $z$ is close to the boundary.

Assuming that $\mu$ is an $\omega$-Hankel measure this yieds
$$\left|\int_{\mathbb{D}}\frac{d\mu(\xi)}{(1-\overline{z}\xi)^{4+2\beta}}\right|\lesssim \frac{\widehat{\omega}(|z|)^{2/p}}{(1-|z|^2)^{2\beta+4-2/p}}.$$

Let now $\sigma(z)=(1-|z|^2)^{2+2\beta}$, and $\eta$ be any radial weight with $\widehat{\eta}>0$. The above estimate then implies that
$$|\mathsf{R}^{\eta,\sigma}P_\eta[\overline{\mu}](z)|\lesssim \frac{\widehat{\omega}(|z|)^{2/p}}{(1-|z|^2)^{2\beta+4-2/p}}.$$

If now $p=2$, by choosing $\eta=\omega \in \mathcal{D}$ and $\beta$ big enough (the reasoning is similar to the earlier reasoning regarding the size of $\beta$), an application of Theorem \ref{bfrac} gives that $P_\omega(\overline{\mu}) \in \mathcal{B}$. 

If on the other hand $p<2$, we let $\eta(z)=W(z)=\widehat{\omega}(z)^{2/p}(1-|z|)^{2/p-2}$ and observe that by \cite{antti1} this implies 

$$\widehat{W}(|z|)=\widehat{\omega}(|z|)^{2/p}(1-|z|^2)^{2/p-1},$$
and that $W\in \mathcal{D}$. As in the previous case, by Theorem \ref{bfrac} this yields for large enough $\beta$ that
$P_W(\overline{\mu})\in \mathcal{B}$. The proof is now complete.

\bigskip

\section{Further remarks and corollaries}

\noindent Since the Hankel form and the Hankel operator are usually generated by an anti-analytic symbol, we mention how these results can be obtained. Let $f$ be analytic and consider $d\mu=\overline{f}d\omega$.

\begin{coro}
Let $ \omega \in \widehat{\mathcal{D}}$, $d\mu=\overline{f}d\omega$ with $f \in A^1_\omega$, and $0<p,q<\infty$, and let $r$ satisfy $1/p+1/q=1/r$. Then $\textbf{H}_{\mu}:A^p_\omega \times A^{q}_\omega \to \mathbb{C}$ is bounded bilinear form if and only if
	
	\begin{itemize}
	
	\item[\rom{1}:]  $f\in A^{r'}_\omega $, if $r>1$.
	
	\item[\rom{2}:]  $f \in{\mathsf{D}}_{\omega} \, BMOA(\infty,\omega)$, if $r=1$.
	
	\item[\rom{3}:]  $\mathsf{R}^{W,\omega} f \in \mathcal{B}$, if $r<1$.
	
	\end{itemize}
	
In \rom{3} the weight $W=W_{1/r,\omega}$ is from \eqref{W}.
\end{coro}

Observe that $\mathsf{R}^{W,\omega}$ should actually be understood as a fractional integral (anti-derivative). Therefore, the space in question could be thought as an analogue of the analytic Lipschitz spaces, see for instance \cite{zhuholom}. In fact, if $\omega(z)=(\alpha+1)(1-|z|^2)^{\alpha}$ is standard weight, then $W_{1/r,\omega}$ is comparable to $(1-|z|^2)^{(\alpha+2)/r-2}$, which recovers the classical result.

Specializing Theorem \ref{TH1b} to the standard weights $\omega_\alpha(z)=(\alpha+1)(1-|z|^2)^\alpha$, we know that for $1<p<\infty$, $H_\mu^{\omega_\alpha} :A^p_{\omega_\alpha} \to A^p_{\omega_\alpha}$ is bounded if and only if 
$$P_{\omega_\alpha}[\overline{\mu}](z)=\int_{\mathbb{D}}\frac{d\overline{\mu}(\xi)}{(1-z\overline{\xi})^{2+\alpha}}(\alpha+1)(1-|\xi|^2)^\alpha dA(\xi)$$ belongs to the Bloch space. For every $t>0$ this is equivalent with 
$$\sup_{z \in \mathbb{D}}(1-|z|^2)^t|\mathsf{R}^{\omega_\alpha,\omega_{\alpha+t}}P_{\omega_\alpha}[\overline{\mu}](z)|<\infty.$$ That is,
$$\sup_{z \in \mathbb{D}}\left|\int_{\mathbb{D}}\frac{(1-|z|^2)^t}{(1-z\overline{\xi})^{2+\alpha+t}}d\overline{\mu}(\xi)\right|<\infty.$$

This is a minor addition to the result of Zhou \cite{zhzh}, which covers the case $t=\alpha+2$ as here any positive $t$ works. Using Theorem \ref{bfrac}, one could also obtain formulas with general weights.

There has been some work on Toeplitz and Hankel operators with symbols outside the space of complex Borel measures. We mention \cite{PTV, RV, TasVir}. For instance, in \cite{RV} Rozenblum and Vasilevski studied Toeplitz operators induced by sesquilinear forms. Motivated by this, if $\Phi:\mathcal{H}(\overline{\mathbb{D}})\times \mathcal{H}(\overline{\mathbb{D}}) \to \mathbb{C}$ is a bilinear form, one can simply define
$$H_\Phi(f,g)=\Phi(f,g),$$
in which case Hankel forms simply agree with the bilinear forms. Alternatively, for a linear form $\tau$ on $\mathcal{H}(\overline{\mathbb{D}})$, one might define
$$H_\tau(f,g)=\tau(fg).$$
With some additional requirements on $\Phi$ and $\tau$ one might be able to obtain a theory similar to that of the classical Hankel forms with very general symbols, for instance distributions or hyperfunctions.



\end{document}